\documentclass[11pt,twoside, leqno]{article}
\usepackage{amsthm}
\usepackage{amssymb}
\usepackage{amsmath}
\usepackage{mathrsfs}
\usepackage{txfonts}
\usepackage{graphics}
\usepackage{epsfig}

\allowdisplaybreaks \def\rr{{\mathbb R}}
\def\rn{{{\rr}^n}}

\def\fz{\infty}

\def\supp{{\mathop\mathrm{\,supp\,}}}
\def\dist{{\mathop\mathrm {\,dist\,}}}
\def\diam{{\mathop\mathrm {\,diam\,}}}
\def\loc{{\mathop\mathrm{\,loc\,}}}

\def\lz{\lambda}

\def\ro{\rho}

\def\gz{{\gamma}}

\def\sz{\sigma}

\def\nz{\nabla}

\def\r{\right}
\def\lf{\left}

\newtheorem{thm}{Theorem}[section]
\newtheorem{lem}{Lemma}[section]
\newtheorem{prop}{Proposition}[section]
\newtheorem{rem}{Remark}[section]
\newtheorem{cor}{Corollary}[section]
\newtheorem{defn}{Definition}[section]
\newtheorem{example}{Example}[section]
\numberwithin{equation}{section}

\begin{document}
\arraycolsep=1pt
\author{Renjin Jiang, Aapo Kauranen and Pekka Koskela} \arraycolsep=1pt
\title{\Large\bf Solvability of the divergence equation implies John \\ via Poincar\'e inequality}
\date{ }
\maketitle
\begin{center}
\begin{minipage}{10cm}\small
{\noindent{\bf Abstract.} Let $\Omega \subset \rr^2$ be a bounded
simply connected domain.
We show that, for a fixed (every) $p\in (1,\fz),$
the divergence equation $\mathrm{div}\,\mathbf{v}=f$ is solvable in
$W^{1,p}_0(\Omega)^2$ for every $f\in L^p_0(\Omega)$,
if and only if $\Omega$ is a John domain, if and only if the weighted
Poincar\'e inequality
$$\int_\Omega|u(x)-u_{\Omega}|^q\,dx\le C\int_\Omega|\nabla u(x)|^q\dist(x,\partial \Omega)^q\,dx$$
holds for some (every) $q\in [1,\fz)$. In higher dimensions similar results
are proved under some additional assumptions on the domain in question.
 }\end{minipage}
\end{center}
\vspace{1.0cm}

\section{Introduction}

\hskip\parindent This paper is devoted to the study of geometric
aspects of the solvability of the divergence equation. Our main tool is a
weighted Poincar\'e inequality.

 Let $\Omega$ be a bounded domain in $\rr^n$, $p\in (1,\fz)$,
and $L^p_0(\Omega)$ be the space of
all functions in $L^p(\Omega)$ which have integral zero over $\Omega$.
The Sobolev space $W^{1,p}(\Omega)$  is defined as
$$W^{1,p}(\Omega):=\lf\{u\in L^p(\Omega):\,\nabla u \in \mathscr{D}'(\Omega)\cap L^p(\Omega)\r\}$$
with the norm
$$\|u\|_{W^{1,p}(\Omega)}:=\|u\|_{L^p(\Omega)}+\|\nabla u\|_{L^p(\Omega)}.$$
The Sobolev space $W^{1,p}_0(\Omega)$ is then defined to be the closure of
smooth functions
with compact support in $\Omega$ under the  $W^{1,p}$-norm.

For $f\in L^p_0(\Omega)$, a vector function $\mathbf{v}=(v_1,\cdots,v_n)\in
L^p(\Omega)^n$ is a
solution to the divergence equation
$\mathrm{div}\,\mathbf{v}=f$, if
$$\int_\Omega \mathbf{v}(x)\cdot \nabla g(x)\,dx=-\int_\Omega f(x)g(x)\,dx$$
holds for each $g\in W^{1,q}(\Omega)$, where $q$ is the H\"older conjugate number of $p$.
We say that {\it the divergence equation with Dirichlet boundary condition} ($\mathrm{div}_{p,0}$, for short)
is solvable, if for each $f\in L^p_0(\Omega)$, there exists
$\mathbf{v}\in W^{1,p}_0(\Omega)^n$ such that
$\mathrm{div}\,\mathbf{v}=f$ holds in the above distributional sense,
and there exists $C>0$, independent of $f$ such that
$$\|\mathbf{v}\|_{W^{1,p}(\Omega)^n}\le C\|f\|_{L^p(\Omega)}.$$


When $\Omega$ has a Lipschitz boundary, it is well known that the divergence
equation ${\mathrm{div}_{p,0}}$ is solvable
for all $p\in (1,\fz)$. There are several ways to prove this result, for
instance, it
can be proved via Functional Analysis, or via elementary constructions; see \cite{asv88,art05,bb03,dl76,ne67}.
Recently, Acosta et al. \cite{adm06} proved that ${\mathrm{div}_{p,0}}$ is solvable on John domains
for all $p\in (1,\fz)$ via a constructive approach.

On the other hand, if $\Omega$ has an external cusp, it is known that
the divergence equation ${\mathrm{div}_{2,0}}$ is not solvable in
$\Omega$; see \cite{adm06}.

Notice that $p=1$ or $p=\fz$, the divergence
equation ${\mathrm{div}_{p,0}}$ does not necessarily admit
a solution in $W^{1,p}(\Omega)^n$  for $p=1$ or for $p=\fz$, even when $\Omega$ is a cube; see \cite{bb03}.


It is natural to ask for necessary geometric conditions  for the solvability
of the divergence equation ${\mathrm{div}_{p,0}}$  for some (all)
$p\in (1,\fz)$.  For domains satisfying a separation property introduced by
Buckley and Koskela \cite{bk95} (see Section 2 for the definition), it was
shown by Acosta et al.
\cite{adm06} that the divergence
equation ${\mathrm{div}_{p,0}}$ is solvable for $p\in (1,n)$, if and only if
$\Omega$ is a John domain.
Our result extends this to the case  $p>n$, and to the case $p=n$ in some
special cases.

Let us first recall the definition of a John domain. This terminology
was introduced  in \cite{ms79}, but these domains were studied already by F.
John \cite{j61}.

\begin{defn}[John domain]
A bounded domain $\Omega\subset \rn$ with a distinguished point
$x_0\in\Omega$ called a John domain if it satisfies the following
``twisted cone" condition: there exists a constant $C>0$ such that
for all $x\in\Omega$, there is a curve $\gz: [0,l]\to\Omega$
parametrised by arclength such that $\gz(0)=x$, $\gz(l)= x_0$, and
$d(\gz(t),\rn\setminus\Omega)\ge Ct$ for all $0\le t\le l.$
\end{defn}

Observe that each Lipschitz domain is a John domain. Moreover,
the boundary of a (planar) John domain may contain an interior cusp, while
exterior cusps are ruled out.

For a mapping $\mathbf{v}=(v_1,v_2,\cdots, v_n)\in W^{1,1}_\loc(\Omega)^n$,
let $D\mathbf{v}$ denote its weak differential. For $x\in \Omega$, we denote
by $\ro(x)$  the distance from $x$
to the boundary $\partial \Omega$, i.e.,  $\ro(x):=\dist(x,\partial \Omega)$.

Our main result is the following theorem.

\begin{thm}\label{t1.1}
Let $\Omega \subset \rr^n$ be a bounded domain that
satisfies the separation property, $n\ge 2$. Then the following conditions
are equivalent:

(i) $\Omega$ is a  John domain;

(ii) for some (every) $p\in (1,n)\cup (n,\fz)$ and each $f\in L^p_0(\Omega)$,
there exists a solution $\mathbf{v}\in W^{1,p}_0(\Omega)^n$
to the equation $\mathrm{div}\,\mathbf{v}=f$ with
$$\|\mathbf{v}\|_{W^{1,p}(\Omega)^n}\le C\|f\|_{L^p(\Omega)};$$

(iii) for some (every) $p\in (1,\fz)$ and each $f\in L^p_0(\Omega)$,
there exists a solution $\mathbf{v}\in W^{1,p}_0(\Omega)^n$
to the equation $\mathrm{div}\,\mathbf{v}=f$ with
$$\lf\|\frac{\mathbf{v}}{\ro}\r\|_{L^p(\Omega)^n}+\|D\mathbf{v}\|_{L^p(\Omega)^{n\times n}}\le C\|f\|_{L^p(\Omega)};$$

(iv) for some (every) $p\in (1,\fz)$ and each $f\in L^p_0(\Omega)$, there
exists a solution $\mathbf{v}\in L^p(\Omega)^n$
to the equation $\mathrm{div}\,\mathbf{v}=f$ with
$$\lf\|\frac{\mathbf{v}}{\ro}\r\|_{L^p(\Omega)^n}\le C\|f\|_{L^p(\Omega)}.$$
\end{thm}

The meaning of ``some (every)'' in the statement above is that the given
existence result for any fixed $p$ in the given parameter range actually
implies the existence for every such $p,$ under the assumptions of the
theorem.

We have not been able to include the case $p=n$ in condition (ii).
The case $p<n$ is proved in \cite{adm06}
by using Sobolev inequalities for $W^{1,p}_0$; our approach for $p>n$ is
based on the fact that solutions in
 $W^{1,p}_0$ satisfy suitable Hardy inequalities.
In Example \ref{e4.1} below, we construct a John domain where
the divergence equation admits a solution in $W^{1,n}_0$, but the Hardy inequalities fail.
However, we can include the case $p=n$ in Theorem \ref{t1.1} (ii) provided
the complement of $\Omega$ is sufficiently thick on $\partial \Omega$; see Theorem 4.1 in Section 4.

Notice that each domain that is quasiconformally equivalent to a uniform
domain $G$ satisfies the separation property. In particular, each simply connected plane domain satisfies the separation property; see \cite{bk95}.

\begin{cor}\label{c1.1}
 Let $\Omega\subset \rr^2$ be a bounded simply connected domain. Then
for some (all) $p\in (1,\fz)$ and each $f\in L^p_0(\Omega)$, there exists a solution $\mathbf{v}\in W^{1,p}_0(\Omega)^2$
to the equation $\mathrm{div}\,\mathbf{v}=f$ such that
$$\|\mathbf{v}\|_{W^{1,p}(\Omega)^{2}}\le C\|f\|_{L^p(\Omega)}, $$
if and only if $\Omega$ is a John domain.
\end{cor}

For $p=2$, by duality, the solvability of  the divergence equation with Dirichlet boundary
condition is  equivalent to the fact
\begin{equation}\label{1.1}
\|f\|_{L^2(\Omega)}\le C\|\nz f\|_{W^{-1,2}_0(\Omega)^2}
\end{equation}
for each $f\in L^2_0(\Omega)$; see \cite{adm06} for instance.
From Corollary \ref{c1.1} it follows that on  a bounded simply connected domain $\Omega\subset \rr^2$,
for each $f\in L^2_0(\Omega) $, \eqref{1.1} holds if and only if $\Omega$ is a John domain.

Our main tool is the equivalence between the John condition and weighted Poincar\'e
inequalities; see Theorem \ref{t2.1} below. To prove that solvability of the divergence
 equation ${\mathrm{div}_{p,0}}$ implies John, Acosta et al. \cite{adm06} used the characterization
of Sobolev-Poincar\'e inequality from Buckley and Koskela \cite{bk95}.
As the Sobolev-Poincar\'e inequality only holds for $p\in [1,n),$ the authors were not able to deal with the case $p\geq n$ in the necessity of the John condition.
To bypass this problem we generalize Buckley and Koskela's characterization
to the weighted setting. Precisely, the following special case of Theorem \ref{t2.1} below says that,
 for a domain $\Omega\subset \rn$ satisfying the separation property,
the weighted Poincar\'e inequality
$$\int_\Omega|u(x)-u_{\Omega}|^p\,dx\le C\int_\Omega|\nabla u(x)|^p\ro(x)^p\,dx$$
holds for some (all) $p\in [1,\fz)$, if and only if $\Omega$ is a John domain.
Using this together with the fact that solutions satisfy Hardy type inequalities, we obtain the desired result.

The paper is organized as follows. In Section 2, we show that our weighted Poincar\'e inequality
implies the John condition. In Section 3, we study the divergence equation on John domains, and
the main result Theorem \ref{t1.1} is proved in Section 4.

 Throughout the paper, we denote by $C$ positive constants which
are independent of the main parameters, but which may vary from
line to line. For $p\in [1,n)$, its Sobolev conjugate $\frac{np}{n-p}$ is denoted by $p^\ast$;
for each $p\in (1,\fz)$, its H\"older conjugate $\frac{p}{p-1}$ is denoted by $p'$.
Corresponding to to a function space $X$, we denote its $n$-vector-valued
analogs by $X^n$.

\section{The weighted Poincar\'e inequality}
\hskip\parindent
In this section, we give a generalization of Buckley and Koskela's characterization from
  \cite{bk95}, which offers us the main tool for proving Theorem \ref{t1.1}.

Let $\Omega\subset \rn$ be a bounded domain, $n\ge 2$. We say that
the  Sobolev-Poincar\'e inequality $(SP_{p,p^\ast}),$  $p\in [1,n)$,
holds if there is a $C>0$ such that for every $u\in C^\fz(\Omega)$ we have that
$$\lf(\int_\Omega |u(x)-u_{\Omega}|^{p^\ast}\,dx\r)^{1/p^\ast}
\le C\lf(\int_\Omega |\nabla u(x)|^p\,dx\r)^{1/p}.\leqno(SP_{p,p^\ast})$$
Above, $u_{\Omega}$ denotes the integral average of
$u$ on $\Omega$, i.e., $u_{\Omega}=\frac{1}{|\Omega|}\int_\Omega u\,dx$.

For $\Omega$ satisfying the separation property (see Definition \ref{d2.1} below),
Buckley-Koskela \cite{bk95} have shown that it is a John domain if and only if $\Omega$ supports a
 Sobolev-Poincar\'e inequality $(SP_{p,p^\ast})$ for some (all) $p\in [1,n).$

Let us first recall the definition of separation property which was introduced in \cite{bk95,bk96}.
Recall that for each $x\in \Omega$, $\ro(x)=d(x,\rn\setminus \Omega)$.

\begin{defn}[separation property]\label{d2.1}
We say that a domain $\Omega \subset\rn$ with a distinguished point
$x_0$ has a separation property if there is a constant $C_s$ such that the following
holds: For each $x\in\Omega$ there is a curve $\gz:[0, 1]\to \Omega$ with $\gz(0)= x$,
$\gz(1)= x_0$, and such that for each $t$ either $\gz([0, t])\subset B:= B(\gz(t), C_s\ro(\gz(t)))$ or
each $y\in \gz([0, t]) \setminus B$ belongs to a different component of $\Omega\setminus \partial B$ than $x_0$.
\end{defn}

It follows from  \cite{bk95} that $\Omega$ has a separation property if it
is quasiconformally equivalent to a uniform domain $G$. In particular, each simply
connected planar domain satisfies a separation property.

\begin{thm}\label{t2.1}
 Let $\Omega\subset \rn$ be a bounded domain satisfying the separation property, $n\ge 2$.
Then the $(P_{p,q,b})$-Poincar\'e inequality holds, i.e., for every $u\in C^\fz(\Omega)$ we have that
$$\lf(\int_\Omega|u(x)-u_{\Omega}|^q\,dx\r)^{p/q}\le C_0\int_\Omega|\nabla u(x)|^p\ro(x)^b\,dx,\leqno(P_{p,q,b})$$
for some (all) $(p,q,b)$ satisfying $1\le p\le q<\fz$, $\frac nq+1-\frac np\ge 0$
and $b=p(\frac nq+1-\frac np)$, if and only if $\Omega$ is a John domain.
\end{thm}

In what follows, we will call $(p,q,b)$ a {\it Sobolev triple}, if $(p,q,b)$
satisfies $1\le p\le q<\fz$, $\frac nq+1-\frac np\ge 0$
and $b=p(\frac nq+1-\frac np)$.

\begin{rem}\rm
Notice that if $(p,q,b)$ is a Sobolev triple, then $b\in [0,p]$. The two endpoint cases of $b$
are of particular interest.

When $b=0$, necessarily $p \in[1,n)$ and $q=p^\ast$;
then $(P_{p,q,b})$-Poincar\'e inequality is the Sobolev-Poincar\'e inequality
$(SP_{p,p^\ast})$.

When $b=p,$ $p$ equals $q$ and takes values in $[1,\fz)$; we then
denote $(P_{p,q,b})$-Poincar\'e inequality  by $(P_{p})$-Poincar\'e inequality for convenience.
The $(P_{p})$-Poincar\'e inequality is the main tool for us to prove
Theorem \ref{t1.1}; see Section 4 below.
\end{rem}

As each simply connected plane domain has a separation property,
the following is an immediate corollary to Theorem \ref{t2.1}.

\begin{cor}\label{c2.1}
 Let $\Omega\subset \rr^2$ be a bounded simply connected domain. Then
the $(P_{p,q,b})$-Poincar\'e inequality holds for some (all) Sobolev triples
 $(p,q,b)$,  if and only if $\Omega$ is a John domain.
\end{cor}

We will need the following characterization of a weighted  Poincar\'e inequality from
Haj{\l}asz and Koskela \cite[Theorem 1]{hk98} (for non-weighted cases see Maz'ya \cite{maz85}).

\begin{thm}[\cite{hk98}]\label{t2.2}
  Let $\Omega$ be a bounded domain in $\rr^n$, $n\ge 2$. Let
$1\le p\le q<\fz$ and  $b\ge 0$. Then the following conditions are equivalent.

(i) There exists a constant $C>0$ such that, for every $u\in C^\fz(\Omega)$ it holds that
$$\lf(\int_\Omega |u(x)-u_{\Omega}|^q\,dx\r)^{1/q} \le C\lf(\int_\Omega |\nabla u(x)|^p\ro(x)^{b}\,dx\r)^{1/p}.$$

(ii) For any fixed  cube $Q\subset\subset \Omega$, there exists a constant $C>0$ such that
$$\lf(\int_\Omega |u(x)|^q\,dx\r)^{1/q} \le C\lf(\int_\Omega |\nabla u(x)|^p\ro(x)^{b}\,dx\r)^{1/p}$$
whenever $u\in C^\fz(\Omega)$ satisfies $u|_Q=0$.
\end{thm}

The following result generalizes \cite[Theorem 2.1]{bk95}
to the setting of a weighted Sobolev-Poincar\'e inequality. Our proof follows the method of \cite{bk95}.

\begin{prop}\label{p2.1}
 Suppose that $(p,q,b)$ is a Sobolev triple,  and that
 $\Omega$ supports a $(P_{p,q,b})$-Poincar\'e inequality.
Fix a ball $B_0\subset \Omega$ and let $w\in \Omega$.
Then there exists a constant $C=C(C_0,n,p,q,\Omega,B_0)$ such that
 $$\diam(T)\le Cd$$
 whenever $T$ is a component of $\Omega\setminus B(w,d)$ that does not intersect $B_0$.
\end{prop}

\begin{proof}
From Theorem \ref{t2.2}, the $(P_{p,q,b})$-Poincar\'e inequality implies that
for each Lipschitz function $u$ that vanishes on $B_0$, it holds that
\begin{equation}\label{2.1}
\lf(\int_\Omega |u(x)|^q\,dx\r)^{p/q} \le C_1\int_\Omega |\nabla u(x)|^p\ro(x)^{b}\,dx,
\end{equation}
where $C_1=C(C_0,n,p,q,\Omega,B_0)$.

Let $T$ be a component of $\Omega\setminus B(w,d)$ that does not intersect $B_0$.
For each $r\ge d$, set $T(r):=T\setminus B(w,r)$; and for all $r>s\ge d$, set
$A(s,r):=T(s)\setminus T(r)$.

If $T(2d)=\emptyset$, then it is obvious $\diam(T)\le 2d$. Otherwise,
$T(2d)\neq\emptyset$ and we continue with following steps.

{\bf Claim 1.} $|T(2d)|\le C_2d^n$, where $C_2=C(C_0,n,p,\Omega,B_0)$. Indeed, set
$$
u(x):=\ \
\begin{cases}
0, \hspace{2cm}  \forall x\in \Omega\setminus T(d);\\
1, \hspace{2cm}  \forall x\in T(2d);\\
\frac{d(x,B(w,d))}{d}, \hspace{0.8cm} \forall x\in A(d,2d).
\end{cases}
$$
Then $u$ is a Lipschitz function that vanishes on $B_0$. The inequality
\eqref{2.1} implies that
\begin{equation}\label{2.2}
|T(2d)|^{p/q}\le\left(\int_\Omega |u(x)|^q\,dx\right)^{\frac{p}{q}} \le
C_1\int_\Omega |\nabla u(x)|^p\ro(x)^{b}\,dx\le
\frac{C_1}{d^p}\int_{A(d,2d)} \ro(x)^{b}\,dx.
\end{equation}
As $B(w,d)$ separates $\Omega$, it follows that
$\ro(x)\le 2d $ for each $x\in B(w,d)$, and hence $\ro(x)\le 4d $ for each $x\in B(w,2d)$.
Thus \eqref{2.2} implies that
\begin{equation*}
|T(2d)|^{p/q}\le 4^b d^{b-p}C_1\int_{A(d,2d)}\,dx\le 4^{b+n}C_1 \omega_n d^{n+b-p},
\end{equation*}
where $\omega_n$ is the volume of the unit ball. As $b=p(\frac nq+1-\frac np)$, Claim 1 follows with
$C_2:=(4^{b+n}C_1 \omega_n)^{q/p}.$

Let $r_0:=2d$, for each $j\ge 1$, choose $r_j>r_{j-1}$ such that
$|T(r_j)|=2^{-j}|T(2d)|$. Thus $|A(r_{j-1},r_j)|=2^{-j}|T(2d)|$.

{\bf Claim 2.} For each $j\ge 1$, $|r_j-r_{j-1}|\le C_32^{-j/n}d$, where $C_3=C(C_0,n,p,q,\Omega,B_0)$.
To prove this, let us consider two cases.

{\bf Case 1.} If there exists $x_j\in A(r_{j-1},r_j)$ such that $\ro(x_j)> C_4 2^{-j/n}d$, where
$C_4:=(6C_2/\omega_n)^{1/n}$, then  $|r_j-r_{j-1}|\le 2C_4 2^{-j/n}d$.

Notice that $T$ is a component of $\Omega\setminus B(w,d)$, and $B(x_j,C_4 2^{-j/n}d)\subset \Omega$ with center
$x_j\in A(r_{j-1},r_j)$. Thus $B(x_j,C_4 2^{-j/n}d)\setminus B(w,r_{j-1})$ is a subset of $T$, which implies that
the set $B(x_j,C_4 2^{-j/n}d)\cap (B(w,r_j)\setminus B(w,r_{j-1}))$ is a subset of $A(r_{j-1},r_j)$.

Suppose towards a contradiction that $|r_j-r_{j-1}|> 2C_4 2^{-j/n}d$. Then as $x_j\in A(r_{j-1},r_j)$, it follows that at least one third of
$B(x_j,C_4 2^{-j/n}d)$ is contained in $A(r_{j-1},r_j)$. We then have
$$|A(r_{j-1},r_j)|\ge \frac 13|B(x_j,C_4 2^{-j/n}d)|\ge \frac13 C_4^n 2^{-j}\omega_n d^n\ge  2^{1-j}C_2d^n
> 2^{-j}|T(2d)|=|A(r_{j-1},r_j)|,$$
which is a contradiction. This implies that $|r_j-r_{j-1}|\le 2C_4 2^{-j/n}d$.

{\bf Case 2.} If for each $x\in A(r_{j-1},r_j)$, it holds that $\ro(x)\le C_4 2^{-j/n}d$,
then $|r_j-r_{j-1}|\le (C_1C_2^{1-\frac pq} C_4^b)^{1/p}  2^{-j/n}d$.

In this case, similarly to Claim 1, we set
$$
u(x):=\ \
\begin{cases}
0, \hspace{2cm}  \forall x\in \Omega\setminus T(r_{j-1});\\
1, \hspace{2cm}  \forall x\in T(r_{j});\\
\frac{d(x,B(w,r_{j-1}))}{r_j-r_{j-1}}, \hspace{0.5cm} \forall x\in A(r_{j-1},r_j),
\end{cases}
$$
and use the  inequality \eqref{2.1} to obtain
\begin{eqnarray*}
|T(r_{j})|^{p/q}&&\le\lf(\int_\Omega |u(x)|^q\,dx\r)^{p/q}
\le C_1\int_{A(r_{j-1},r_j)} \frac{\ro(x)^{b}}{|r_j-r_{j-1}|^p}\,dx \\
&&\le \frac{C_1C_4^b 2^{-jb/n}d^b}{|r_j-r_{j-1}|^p}|A(r_{j-1},r_j)|,
\end{eqnarray*}
which together with the facts $q\ge p$, $b=p(\frac nq+1-\frac np)$ and
$|T(r_{j})|=|A(r_{j-1},r_j)|=2^{-j}|T(2d)|\le 2^{-j}C_2d^n$ implies that
\begin{eqnarray*}
|r_j-r_{j-1}|^p&&\le C_1C_4^b 2^{-jb/n}d^b|A(r_{j-1},r_j)|^{1-\frac pq}\le
C_1C_4^b 2^{-j\frac bn-j(1-\frac pq) }d^{b+n(1-\frac pq)}C_2^{1-\frac pq}\\
&&\le C_1C_2^{1-\frac pq} C_4^b 2^{-jp/n}d^{p}.
\end{eqnarray*}
Hence, Claim 2 follows  with $C_3:=\max\{2C_4, (C_1C_2^{1-\frac pq} C_4^b)^{1/p} \}.$
Moreover, notice that $C_1$, $C_2$, $C_4$ and hence $C_3$
depend only on  $C_0,n,p,q,\Omega,B_0$.

By using Claim 2, we finally obtain that
$$\diam (T)\le 2d+\sum_{j\ge 1} |r_j-r_{j_1}|\le Cd,$$
where $C=C(C_0,n,p,q,\Omega,B_0)$, which completes the proof.
\end{proof}

\begin{proof}[Proof of Theorem \ref{t2.1}]
If $\Omega$ is a John domain, then from \cite[Theorem 2.1]{km00}
it follows that the $(P_{p,q,b})$ holds for all Sobolev triples $(p,q,b)$
satisfying $1\le p\le q<\fz$, $\frac nq+1-\frac np\ge 0$
and $b=p(\frac nq+1-\frac np)$; also see  \cite{hk98}.

For the converse we employ the argument from
\cite[Proof of Theorem 1.1]{bk95} via Proposition \ref{p2.1}
We sketch the proof for the sake of completeness.

Suppose that $(P_{p,q,b})$ holds for a Sobolev triple $(p,q,b)$.
Fix  $x\in\Omega$ and pick a curve $\gz:[0, 1]\to \Omega$ with $\gz(0)= x$,
$\gz(1)= x_0$ as in Definition \ref{d2.1}. According to \cite[pp. 385-386]{ms79}
and \cite[pp. 7-8]{nv91}, it is enough to show that $\diam(\gz([0,t]))\le C\ro(\gz(t))$.

Let $C_s$ be a constant as in Definition \ref{d2.1}.
If $\gz([0,t])\subset B(\gz(t), C_s\ro(\gz(t)))$, the conclusion
is obvious.

Otherwise the separation property implies that
$\partial B:=\partial B(\gz(t), C_s\ro(\gz(t)))$ separates
$\gz([0, t])\setminus B$ from $x_0$. Let us consider two cases.

{\bf Case 1.} If $B\cap B_0\neq \emptyset$, where $B_0:=B(x_0, \ro(x_0)/2 )$, then
$\gz([0,t])\subset B(\gz(t), C_5\ro(\gz(t)))$ with $C_5=\frac{2\diam(\Omega)}{\ro(x_0)}C_s $.
Indeed, as $B\cap \partial \Omega\neq \emptyset$ and $B\cap B(x_0, \ro(x_0)/2 ) \neq \emptyset$,
it follows  $C_s\ro(\gz(t))\ge \ro(x_0)/2$. Hence
$\gz([0,t])\subset B(\gz(t), \diam (\Omega))\subset  B(\gz(t), C_5\ro(\gz(t)))$.

{\bf Case 2.} If $B\cap B_0=\emptyset$, then
$\gz([0,t])\subset B(\gz(t), C_6\ro(\gz(t)))$, where $C_6$ depends only on
$C_0,n,p,q,\Omega,B_0, C_s$. Let $T$ be the component containing $\gz([0, t])\setminus B$.
As $B$ separates $\gz([0, t])\setminus B$ from $x_0$, $T$ is a component of
$\Omega\setminus B$ that does not intersect $B_0$. By using Proposition \ref{p2.1},
we see that $\gz([0,t])\subset B(\gz(t), C_6\ro(\gz(t)))$.

By letting $C=\max\{C_s,C_5,C_6\}$, we obtain $\diam(\gz([0,t]))\le C\ro(\gz(t))$,
which completes the proof.
\end{proof}

\section{The divergence equation}
\hskip\parindent
In this section, we study the divergence equation on John domains.

\begin{thm}\label{t3.1}
Let $\Omega$ be a  John domain in $\rr^n$, $n\ge 2$, and let $q\in (1,\fz)$.
Then for each $f\in L^q_0(\Omega)$, there exists a solution
$\mathbf{u}\in W^{1,q}_0(\Omega)^n$ to the equation $\mathrm{div}\,\mathbf{u}=f$ in $\Omega$.
Moreover, there exists a constant $C>0$, independent of $f$, such that
\begin{equation}\label{3.1}
\lf\|\frac{\mathbf{u}}{\ro}\r\|_{L^{q}(\Omega)^n}+\|D \mathbf{u}\|_{L^q(\Omega)^{n\times n}}\le C\|f\|_{L^q(\Omega)}.
\end{equation}
\end{thm}

\begin{rem}\label{r3.1}\rm
Notice that on a bounded domain $\Omega$, for $q>n$, if $\mathbf{u}\in W^{1,q}_0(\Omega)^n$,
then from the Hardy inequality $(H_q)$  it follows that
$$\lf\|\frac{\mathbf{u}}{\ro}\r\|_{L^{q}(\Omega)^n}\le
C\|D \mathbf{u}\|_{L^q(\Omega)};$$
see Lemma \ref{l4.1} below. Thus the case $q>n$ in Theorem \ref{t3.1}
follows directly from Acosta et al. \cite{adm06}.

However, for $q\le n$, the Hardy inequality may fail even on a
John domain. For instance, the domain $B(0,1)\setminus \{0\}$ does not admit the
$n$-Hardy inequality, but it is a John domain; see \cite{kl09}.
Thus the main improvement in Theorem \ref{t3.1} is that for $q\in(1,n]$, there are
solutions $\mathbf{u}$  belong to $W^{1,q}_0(\Omega)^n$
and satisfying \eqref{3.1}.
\end{rem}

For the proof, we need the following geometric decomposition from \cite{drs10}.

\begin{prop}[\cite{drs10}]\label{p3.1}
Let $\Omega$ be a  John domain in $\rr^n$, $n\ge 2$.
Let $\{Q_j\}_j$ be a Whitney decomposition of $\Omega$.
Then there exists $\sz>1$ and a family of linear operators
$\{T_j\}_{j}$ such that for all $p\in (1,\fz)$ and all $f\in L^p_0(\Omega)$:

(i) $\sum_j \chi_{\sz Q_j}\le C\chi_{\Omega}$;

(ii) $\supp T_jf\subset \sz Q_j$ and $T_jf\in L^p_0(\sz Q_j)$;

(iii) $f =\sum_{j\in I} T_jf$ in $L^p(\Omega)$;

(iv) $\sum_{j\in I} \int_{2Q_j}|T_jf(x)|^p\,dx\le C\int_\Omega |f(x)|^p\,dx$
for some $C=C(p,n)>0$.
\end{prop}

\begin{rem}\rm
Notice that Duran et al.  \cite{dmrt10} also give an atomic decomposition via
functional analysis; while the decomposition in Proposition
\ref{p3.1} uses the geometric structure of $\Omega$, and does not depend on $p$.
\end{rem}

\begin{proof}[Proof of Theorem \ref{t3.1}]
As discussed in Remark \ref{r3.1}, we only need to consider the case $q\le n$.

Suppose $f\in L^q_0(\Omega)$. We may choose a sequence $\{f_k\}_{k=1}^\fz\in
L^q_0(\Omega)\cap L^\fz(\Omega)$ such that $f_k\to f$ in $L^q_0(\Omega)$.
Let $\{Q_j \}_j$ be a Whitney covering of $\Omega$ as in Proposition \ref{p3.1}.
Applying Proposition \ref{p3.1} to
each $f_k$, we see that $f_k=\sum_j T_jf_k$, where the decomposition holds
in both $L^q_0(\Omega)$ and $L^{2n}_0(\Omega)$. The same conclusion holds for  $f=\sum_j T_jf$  in $L^q_0(\Omega).$

By using \cite[Theorem 2]{bb03} (see also \cite[Theorem 5.2]{drs10}),
on each cube $Q\subset \rn$, there exists a linear operator $S$ that maps
$L^p_0(Q)$ into $W^{1,p}_0(Q)$ for all $p\in (1,\fz)$, such that for
each $g\in L^p_0(Q)$, $\mathrm{div}\,(S g)=g$ and
\begin{equation*}
\|D (S g)\|_{L^p(Q)}\le C(n,q)\|g\|_{L^p(Q){n\times n}}.
\end{equation*}
Thus, by a translation and scaling argument,
it follows that for each $j$, there exist a linear operator
$S\!_j$ that maps $L^p_0(\sz Q_j)$ into $W^{1,p}_0(\sz Q_j)^n$ for each $p\in (1,\fz)$,
and so that $\mathrm{div}S\!_j T_j g= T_j g$  and
\begin{equation*}
\|DS\!_j T_jg\|_{L^p(\sz Q_j){n\times n}}\le C(n,p)\|T_jg\|_{L^p(\sz Q_j)}
\end{equation*}
for every $g\in L^p_0(\Omega)$.

Write $\mathbf{u}(x):=\sum_{j\in I}S\!_jT_jf(x)$ and $\mathbf{u}_k(x):=\sum_{j\in I}S\!_jT_jf_k(x)$.
As $\sum_{j\in I} \chi_{\sigma Q_j}\leq C \chi_{\Omega},$  we
have $\mathbf{u}, \mathbf{u}_k\in W^{1,q}(\Omega)^n$, with
\begin{eqnarray*}
  \int_\Omega |D\mathbf{u}(x)|^q\,dx&&
  \le C\sum_{j\in I}\int_{\sz Q_j}|D(S\!_jT_j)f(x)|^q\,dx\\
  &&\le C\sum_{j\in I}\int_{\sz Q_j}|T_jf(x)|^q\,dx\le C\int_\Omega|f(x)|^q\,dx.
\end{eqnarray*}
 Moreover,
\begin{eqnarray*}
  \int_\Omega \frac{|\mathbf{u}(x)|^q}{\ro(x)^q}\,dx&&
  \le C\sum_{j\in I}\ell(Q_j)^{-q}\int_{2Q_j}|S\!_jT_jf(x)|^q\,dx\\
  &&\le C\sum_{j\in I}\int_{2Q_j}|T_jf(x)|^q\,dx\le C\int_\Omega|f(x)|^q\,dx.
\end{eqnarray*}
The above two estimates prove \eqref{3.1}.

It remains to show that $\mathbf{u}(x)\in W^{1,q}_0(\Omega)^n$. Since $f_k\in L^\fz(\Omega)$,
 the Sobolev embedding theorem ensures  that
\begin{eqnarray*}
\|S\!_jT_jf_k\|_{L^\fz(\sz Q_j)}\le C\ell(Q_j)^{1/2}\|T_jf_k\|_{L^{2n}(2Q_j)}\le C\ell(Q_j)^{1/2}\|f_k\|_{L^{2n}(\Omega)},
\end{eqnarray*}
and hence, $|\mathbf{u}_k(x)|\le  C\ro(x)^{1/2}\to 0$ as $x\to \partial \Omega$, which implies that $\mathbf{u}_k\in W^{1,q}_0(\Omega)^n$.

As $f_k\to f$ in $L^q_0(\Omega)$, we finally obtain
\begin{eqnarray*}
\|\mathbf{u}_k-\mathbf{u}\|^q_{W^{1,q}(\Omega)}&&\le C\sum_j
\|S\!_jT_jf_k-S\!_jT_jf\|^q_{W^{1,q}(\Omega)}\le C\sum_j \|T_j(f_k-f)\|^q_{L^q_0(\Omega)}\\
&&\le C\|f_k-f\|_{L^q_0(\Omega)}\to 0,
\end{eqnarray*}
as $k\to\fz$. Thus $\mathbf{u}_k\in W^{1,q}_0(\Omega)$
implies $\mathbf{u}\in W^{1,q}_0(\Omega)$. The proof is complete.
\end{proof}


\section{Proof of Theorem \ref{t1.1}}
\hskip\parindent
In this section, we prove Theorem \ref{t1.1}.
We need the following Hardy inequality; see \cite{an86,ha99,kl09} for instance.
\begin{lem}[Hardy inequality]\label{l4.1}
Let $\Omega$ be a bounded domain in $\rn$, $n\ge 2$.
If $p>n$, then there exists $C>0$ such that the Hardy inequality $(H_{p})$ holds
for every $v\in W^{1,p}_0(\Omega)$,
\begin{equation}
\label{hardy}
\int_\Omega\frac{|v(x)|^p}{\ro(x)^p}\,dx\le C\int_\Omega|\nabla v(x) |^p\,dx. \tag{$H_{p}$}
\end{equation}
\end{lem}
%
%

\begin{proof}[Proof of Theorem \ref{t1.1}]
Given a John domain $\Omega$, from \cite{adm06} it follows that (ii) holds;
from Theorem \ref{t3.1} it follows that (iii) holds and hence (iv) holds.

Conversely, suppose that separation property holds on $\Omega$.
Let us first show that (ii) implies (i).

In the case $p\in (1,n)$,  it follows from \cite{adm06}  that $\Omega$ is a John
domain.
Suppose now $p\in (n,\fz)$.
Thus, for each $f\in L^p_0(\Omega)$, there is
$\mathbf{u}\in W^{1,p}_0(\Omega)^n$ satisfying $\mathrm{div}\, \mathbf{u}=f$ and
$$\|D\mathbf{u}\|_{L^p(\Omega)^{n\times n}}\le C\|f\|_{L^p(\Omega)}.$$

Applying the Hardy inequality \eqref{hardy} to $\mathbf{u}\in W^{1,p}_0(\Omega)^p$, $p>n$, we see that
$$\int_\Omega\frac{|\mathbf{u}(x)|^p}{\ro(x)^p}\,dx\le c\int_\Omega|D\mathbf{u}(x)|^p\,dx
\le C\int_\Omega|f(x)|^p\,dx.$$

Next, for each $u\in  W^{1,p'}(\Omega)$ and each $f\in L^{p}(\Omega)$, where $1/p'+1/p=1$, it follows that
\begin{eqnarray}\label{4.1}
  \lf|\int_\Omega f(x)(u(x)-u_{\Omega})\,dx\r|&&=\lf|\int_\Omega (f(x)-f_\Omega)(u(x)-u_{\Omega})\,dx\r|
  \\&&=\lf|\int_\Omega \mathbf{u(x)}\cdot \nabla (u(x)-u_{\Omega})\,dx\r|\nonumber\\
  &&\le \lf(\int_\Omega\frac{|\mathbf{u}(x)|^p}{\ro(x)^p}\,dx\r)^{1/p}
  \lf(\int_\Omega|\nabla u(x)|^{p'}\ro(x)^{p'}\,dx\r)^{1/p'}\nonumber\\
  &&\le C\lf(\int_\Omega|f(x)|^p\,dx\r)^{1/p}
  \lf(\int_\Omega|\nabla u(x)|^{p'}\ro(x)^{p'}\,dx\r)^{1/p'}.\nonumber
  \end{eqnarray}
Taking the supremum over the set $\{f\in L^p(\Omega):\,\|f\|_{L^p(\Omega)}\le 1\}$, we see that
$$\int_\Omega|u(x)-u_{\Omega}|^{p'}\,dx\le C\int_\Omega|\nabla u(x)|^{p'}\ro(x)^{p'}\,dx,
$$
i.e., the $(P_{p'})$-Poincar\'e inequality holds on $\Omega$.
By using Theorem \ref{t2.1} we see that $\Omega$ is a John domain.

Let us show that (iv) implies (i), which further implies that (iii) implies (i).
(iv) implies that for each $f\in L^p_0(\Omega)$, there is
$\mathbf{u}\in L^p(\Omega)^n$ satisfying $\mathrm{div}\,\mathbf{u}=f$ and
$$\lf\|\frac{\mathbf{u}}{\ro}\r\|_{L^p(\Omega)^n}\le C\|f\|_{L^p(\Omega)}.$$
Then, using a duality argument as in \eqref{4.1}, it follows the $(P_{p'})$-Poincar\'e inequality holds on $\Omega$.
Using Theorem \ref{t2.1} again, we see that $\Omega$ is a John domain, which completes the proof.
\end{proof}

Notice Theorem \ref{t1.1} (ii) does not cover the borderline case $p=n$.
For $1<p<n,$ a calculation similar to the one in the
proof of Theorem \ref{t1.1} was done in \cite{adm06} relying on a
Sobolev-Poincar\'e inequality.
This does not work for $p\ge n$ and we use Hardy inequality to bypass
the problem in the case $p>n$. In the case $p=n$ we
cannot rely on such an inequality without additional assumptions.

We can include the case $p=n$ in Theorem \ref{t1.1} (ii) provided
the complement of $\Omega$ is sufficiently thick on $\partial \Omega$. Precisely,
it suffices to assume there exists $\lz>0$ such that
$\mathscr{H}^\lz_\fz(\Omega^c\cap B(w,r))\ge Cr^\lz$ for all $w\in \partial \Omega$ and $r>0$.
Here $\mathscr{H}^\lz_\fz$ denotes $\lz$-dimensional Hausdorff content; see \cite{kl09}.
For example, each simply connected plane domain satisfies this condition.

\newtheorem{thmnew}{Theorem 4.1}
\renewcommand\thethmnew{}

\begin{thmnew}\label{thmnew}
Let $\Omega \subset \rr^n$ be a bounded domain
satisfying the separation property, $n\ge 2$. Suppose that
there exists $\lz>0$ such that
$\mathscr{H}^\lz_\fz(\Omega^c\cap B(w,r))\ge Cr^\lz$ for all $w\in \partial \Omega$ and $r>0$.

Then $\Omega$ is a John domain if and only if
for some (all) $p\in (1,\fz)$ and each $f\in L^p_0(\Omega)$, there exists a solution $\mathbf{v}\in W^{1,p}_0(\Omega)^n$
to the equation $\mathrm{div}\,\mathbf{v}=f$ with
$$\|\mathbf{v}\|_{W^{1,p}(\Omega)^n}\le C\|f\|_{L^p(\Omega)^n}.$$
\end{thmnew}
\begin{proof} By Theorem \ref{t1.1}, we only need to show that if
the divergence equation $\mathrm{div}_{n,0}$ is solvable, then $\Omega$
is a John domain. In this case, from \cite{l88,kl09}, it follows that
every $v\in W^{1,n}_0(\Omega)$, there is a constant $C$ such that
$$\int_\Omega\frac{|v(x)|^n}{\ro(x)^n}\,dx\le C\int_\Omega|\nabla v(x) |^n\,dx. \leqno(H_{n})$$
Thus, for each $f\in L^n_0(\Omega)$, there exists a solution $\mathbf{v}\in W^{1,n}_0(\Omega)^n$
to the equation $\mathrm{div}\,\mathbf{v}=f$ such that
$$\lf(\int_\Omega\frac{|\mathbf{v}(x)|^n}{\ro(x)^n}\,dx\r)^{1/n}\le c\|\mathbf{v}\|_{W^{1,n}(\Omega)^n}\le C\|f\|_{L^n(\Omega)^n}.$$
Arguing as in \eqref{4.1}, we see that $(P_{\frac{n}{n-1}})$-Poincar\'e inequality
holds on $\Omega$, which implies $\Omega$ is a John domain by Theorem \ref{t2.1}.
\end{proof}

On the other hand, we have the following example.


%

\begin{example}\label{e4.1}\rm
 For each $1<p\le n$, there is John domain $\Omega$ that satisfies the separation property, $f\in L^p_0(\Omega)$ and $\mathbf v\in W^{1,p}_0(\Omega)$
so that $\mathrm{div}\,\mathbf{v}=f$, and
$\|\frac{\mathbf{v}}{\ro}\|_{L^p(\Omega)^n}=\fz$.
\end{example}

For simplicity, we only consider the case $n=2;$ our reasoning easily extends
to cover the higher dimensional case.
Let $p\in (1,2]$ and set $\Omega :=B^2(0,2)\setminus E,$ where $E\subset [0,1]$ is a compact set so that
 $\mathscr{H}^{2-p}(E)<\infty,$ but $\int_{B^2(0,1)} d(x,E)^{-p}dx=\infty.$ Fix $\varphi \in C_0^{\infty}(B(0,2))$
 with $\varphi(x)=x_1$ on $B^2(0,1).$ Then $\mathbf{v}=\nabla \varphi\in W^{1,p}_0(B^2(0,2))^2$  is a solution to $\mathrm{div}\,\mathbf v =\Delta \varphi$
 on $B^2(0,2),$ and in particular,  on $\Omega=B^2(0,2)\setminus E.$ Moreover, $\int_{\Omega} |\frac{\mathbf{v}}{\rho}|^p\,dx=\infty$
 and it is easy to check that  $v\in W^{1,p}_0(\Omega)$ via $\mathscr{H}^{2-p}(E)<\infty,$ and that $\Omega$ satisfies the separation property.


\begin{proof}[Proof of Corollary \ref{c1.1}]
It was proved in \cite{bk95} that each simply connected plane domain satisfies the separation property.
Moreover, it is trivial that for each simply connected plane domain $\Omega$,
$\mathscr{H}^1_\fz(\Omega^c\cap B(w,r))\ge Cr$ for all $w\in \partial \Omega$ and $r>0$.
Hence, $\Omega$ satisfies the requirements for Theorem 4.1, and Corollary \ref{c1.1} follows.
\end{proof}

\subsection*{Acknowledgment}
Jiang and Koskela were supported by the Academy of Finland Grants 131477 and 263850 and Kauranen was supported by The
Finnish National Graduate School in Mathematics and its Applications.

\vspace{0.4cm}

\noindent Renjin Jiang$^{1}$, Aapo Kauranen$^{2}$ \& Pekka Koskela$^{2}$

\

\noindent
1. School of Mathematical Sciences, Beijing Normal University,
Laboratory of Mathematics and Complex Systems, Beijing 100875, People's Republic of China

\

\noindent 2. Department of Mathematics and Statistics, University of Jyv\"{a}skyl\"{a}, P.O. Box 35 (MaD),
FI-40014
Finland

\

\noindent{\it E-mail addresses}:
\texttt{rejiang@bnu.edu.cn}

\hspace{2.3cm}
\texttt{aapo.p.kauranen@jyu.fi}

\hspace{2.3cm}
\texttt{pkoskela@maths.jyu.fi}


\begin{thebibliography}{999}

\bibitem{adm06} Acosta, G., Dur\'an, R.G., Muschietti, M.A.,
Solutions of the divergence operator on John domains, Adv. Math. 206 (2006), 373-401.

\vspace{-0.3cm}
\bibitem{an86} Ancona, A., On strong barriers and inequality of Hardy for
domains in $\rn$, J. London Math. Soc. 34 (1986), 274C290.

\vspace{-0.3cm}
\bibitem{asv88} Arnold, D.N., Scott, L. R., Vogelius, M., Regular inversion of the divergence
operator with Dirichlet boundary conditions on a polygon,
Ann. Scuola Norm. Sup. Pisa Cl. Sci. (4) 15 (1988), 169-192.

\vspace{-0.3cm}
\bibitem{art05} Auscher, P., Russ, E., Tchamitchian, P., Hardy Sobolev spaces on
strongly Lipschitz domains of $\rn$, J. Funct. Anal. 218 (2005), 54-109.

\vspace{-0.3cm}
\bibitem{bb03}  Bourgain, J., Brezis, H., On the equation $\mathrm{div}{Y}=f$ and
application to control of phases, J. Amer. Math. Soc. 16 (2003), 393-426.

\vspace{-0.3cm}
\bibitem{bk95} Buckley, S., Koskela, P., Sobolev-Poincar\'e implies
John, Math. Res. Lett. 2 (1995), 577-593.

\vspace{-0.3cm}
\bibitem{bk96} Buckley, S., Koskela, P., Criteria for imbeddings of
Sobolev-Poincar\'e type, Internat. Math. Res. Notices (1996), 881-901.

\vspace{-0.3cm}
\bibitem{drs10} Diening L., Ru$\check{\rm z}$i$\check{\rm c}$ka M., Schumacher, K.,
A decomposition technique for John domains, Ann.
Acad. Sci. Fenn. Math., 35 (2010), 87-114.


\vspace{-0.3cm}
\bibitem{dmrt10}
Dur\'an, R.G., Muschietti, M.A., Russ, E., Tchamitchian, P.,
Divergence operator and Poincar\'e inequalities on arbitrary bounded
domains, Complex Var. Elliptic Equ. 55 (2010), 795-816.


\vspace{-0.3cm}
\bibitem{dl76}
Duvaut G., Lions J.-L., Inequalities in Mechanics and Physics, Springer, 1976.

\vspace{-0.3cm}
\bibitem{f47} Friedrichs K.O., On the boundary-value problems of the theory of
elasticity and Korn's inequality, Ann. of Math. 48 (2) (1947) 441-471.


\vspace{-0.3cm}
\bibitem{ha99} Haj{\l}asz P., Pointwise Hardy inequalities,
Proc. Amer. Math. Soc. 127 (1999), 417-423.

\vspace{-0.3cm}
\bibitem{hk98} Haj{\l}asz P., Koskela P., Isoperimetric inequalities and
imbedding theorems in irregular domains, J. London Math. Soc. (2) 58
(1998), 425-450.




\vspace{-0.3cm}
\bibitem{j61} John F., Rotation and strain, Comm. Pure Appl. Math. 4 (1961) 391C414.

\vspace{-0.3cm}
\bibitem{km00} Kilpel\"ainen T., Mal\'y J., Sobolev inequalities on sets with
irregular boundaries, Z. Anal. Anwendungen 19 (2000), 369-380.

\vspace{-0.3cm}
\bibitem{kl09} Koskela, P., Lehrb\"ack, J., Weighted pointwise Hardy inequalities,
J. Lond. Math. Soc. (2) 79 (2009), 757-779.


\vspace{-0.3cm}
\bibitem{l88} Lewis, J.L., Uniformly fat sets, Trans. Amer. Math. Soc. 308 (1988), 177-196.

\vspace{-0.3cm}
\bibitem{ms79} Martio, O., Sarvas, J., Injectivity theorems in plane and space,
Ann. Acad. Sci. Fenn. Ser. A I Math. 4 (1979), 383-401.

\vspace{-0.3cm}
\bibitem{ne67} J. Ne$\check{\rm c}$as, Les m\'ethodes directes en th\'eorie des \'equations elliptiques.
(French) Masson et Cie (Eds.), Paris; Academia, Editeurs, Prague, 1967.


\vspace{-0.3cm}
\bibitem{nv91}  N\"akki, R., V\"ais\"al\"a, J., John disks,
Exposition. Math. 9 (1991), 3-43.

\vspace{-0.3cm}
\bibitem{te84} Temam, R., Navier-Stokes equations. Theory and numerical analysis. Third edition.
North-Holland Publishing Co., Amsterdam, 1984.



\vspace{-0.3cm}
\bibitem{maz85} Maz'ya V.G., Sobolev spaces (Springer, 1985).


\end{thebibliography}
\end{document}